\documentclass{amsart} 

\newcommand{\N}{\mathbb{N}}
\usepackage{amssymb}

\newcommand \seq[1]{{\left\langle{#1}\right\rangle}}
 
\newcommand{\uhr}[1]{\! \upharpoonright_{#1}}

\newcommand{\cero}{\mathbf{0}}
\newcommand{\converge}{\!\!\downarrow}

\newcommand{\w}{\omega}

\renewcommand \phi{\varphi}

\newcommand{\comment}[1]{}
\newcommand{\pz}{$\Pi^0_1$\ }

\newcommand{\WW}{\mathcal{W}}
\newcommand{\RR}{\mathcal{R}}
\newcommand{\DD}{\mathcal{D}}

\theoremstyle{plain}
\newtheorem{theorem}{Theorem}
\newcounter{claimCounter}[theorem]

\newtheorem{proposition}[theorem]{Proposition}

\newtheorem{question}[theorem]{Question} 
\newtheorem{claim}[claimCounter]{Claim}

\theoremstyle{definition}
\newtheorem{definition}[theorem]{Definition} 

\theoremstyle{remark}

\begin{document}

\title[Realizing Computably Enumerable Degrees in Separating Classes]
{Realizing Computably Enumerable Degrees in Separating Classes}

\author{Peter Cholak}
\address{Department of  Mathematics\\
University of Notre Dame \\
South Bend, Indiana USA}
\email{cholak@nd.edu}
\thanks{}

\author{Rod Downey, Noam Greenberg,  and Daniel Turetsky}
\address{School of Mathematics and  Statistics\\
Victoria University of Wellington\\
PO Box 600, Wellington, New Zealand}
\email{Rod.Downey,noam.greenberg,dan.turetsky@vuw.ac.nz}
\thanks{}

\begin{abstract}
	We investigate what collections of c.e.\ Turing degrees can be realised as the collection of elements of a separating $\Pi^0_1$ class of c.e.\ degree. We show that for every c.e.\ degree $\mathbf{c}$, the collection $\{\mathbf{c}, \mathbf{0}'\}$ can be thus realized. We also rule out several attempts at constructing separating classes realizing a unique c.e.\ degree. For example, we show that there is no \emph{super-maximal} pair: disjoint c.e.\ sets~$A$ and~$B$ whose separating class is infinite, but every separator of c.e.\ degree is a finite variant of either~$A$ or $\overline{B}$. 
\end{abstract}

\thanks{The second author is partially supported by Marsden Fund of New Zealand. The first author is partially supported by  NSF-DMS-1854136, a FRG grant from the NSF.  We dedicate this paper to Hugh Woodin and Ted Slaman as part of their combined birthday celebrations.}
\maketitle

\date{\today}

\maketitle


\section{Introduction}
This paper is concerned with (computably bounded) \pz classes.
Of course we can consider these (up to Turing degree) 
as being a collection of infinite paths through a computable binary tree.
They have deep connections with computability theory in general,
as well as reverse mathematics, algorthmic randomness and many other areas.
See, for example, Cenzer and Jockusch \cite{CJ}.

The meta-question we attack concerns realizing c.e.\ degrees as 
members of \pz classes.
In this regard, we follow some earlier studies of 
Csima, Downey and Ng \cite{CDN} and Downey and Melnikov \cite{DM},
but in this case we will be looking at \emph{separating classes}.

Recall that 
one of the fundamental theorems in this area is the Computably Enumerable 
 Basis Theorem 
which says that each \pz class has a member of computably enumerable 
degree. Indeed, if $\alpha$ is the left- or right-most path of a 
\pz class $P$, then there is a c.e.\ set $W_e$ such that $W_e\equiv_T \alpha$.

\begin{definition}[\cite{CDN}] 
We will say that a c.e. degree ${\mathbf w}$
is \emph{realized} in a \pz class $P$ iff there exists some $\beta\in P$
with deg$_T(\beta)={\mathbf w}$.
\end{definition}

The fundamental question attacked in \cite{CDN}
was ``What sets of c.e.\ degrees can be realized in 
a \pz class?'' 
In \cite{CDN},
Csima, Downey and Ng   give a surprising 
 characterization of the sets of c.e.\ degrees that can be realized.
They showed that the question is related to one of 
representing index sets. 

For a set $\WW$ of c.e.\ degrees, let 
\[
	I(\WW) = \left\{ e \,:\,  W_e\in \WW \right\}
\]
be the set of indices of c.e.\ sets whose degrees are in~$\WW$. Letting $\RR$ be the collection of all c.e.\ degrees, for a \pz class~$P$, let 
\[
	\WW[P] = \left\{ \mathbf{w}\in \RR  \,:\,  \mathbf{w}\text{ is realized in }P  \right\}.
\]
A calculation shows: 
\begin{proposition}[\cite{CDN}] \label{indexpi4}
For any \pz class~$P$, the index set $I(\WW[P])$ is $\Sigma^0_4$.
\end{proposition}

The first natural question is whether Proposition~\ref{indexpi4} reverses, that is, if every $\Sigma^0_4$ index set is realized as the collection of indices of c.e.\ sets whose degrees are realized in some \pz class. The answer is negative (see a discussion in~\cite{CDN}). To describe which sets can be realized, we observe that collections of degrees whose index sets are complicated can nonetheless have simple \emph{representations}, in the following sense. For a set $S\subseteq \N$, let 
\[
	\DD(S) = \left\{ \deg_T(W_e) \,:\,  e\in S \right\}
\]
be the collection of degrees of c.e.\ sets whose indices are in~$S$. Since Turing equivalence is $\Sigma^0_3$, if $S$ is $\Sigma^0_4$, then so is its closure under Turing equivalence, namely, $I(\DD(S))$ (denoted by $G(S)$ in~\cite{CDN}), is also $\Sigma^0_4$. But it is possible for~$S$ to be very simple but $I(\DD(S))$ to be complicated; for example, $I(\{\mathbf{0}'\})$ is $\Sigma^0_4$-complete, but of course $\{\cero'\} = \DD(S)$ where~$S$ is a singleton. For a complexity class~$\Gamma$, we say that a set~$\WW$ of c.e.\ degrees is \emph{$\Gamma$-representable} if $\WW = \DD(S)$ for some set~$S$ in the class~$\Gamma$. This notion gives the following characterization:

\begin{theorem}[Csima, Downey and Ng \cite{CDN}] \label{thm:CDN}
	The following are equivalent for a set~$\WW$ of c.e.\ degrees:
	\begin{itemize}
		\item[(i)] $\WW$ is $\Sigma^0_3$-representable;
		\item[(ii)] $\WW$ is computably representable; 
		\item[(iii)] $\WW = \WW[P]$ for some \pz class~$P$. 
	\end{itemize}
\end{theorem}
They also show that the class~$P$ in (iii) can be taken to be perfect. As mentioned, not every set of c.e.\ degrees whose index-set is $\Sigma^0_4$ is thus realized; in~\cite{CDN}, the authors give an example of a $\Pi^0_3$-represented set of c.e.\ degrees which contains~$\cero$ but is not realized as the collection of c.e.\ degrees of elements of some \pz class.

\smallskip

Theorem~\ref{thm:CDN} implies that, for example,
the superlow c.e.\ degrees,  the $K$-trivial c.e.\ degrees, and all upper cones
are realizable.
Lower cones were classified by Downey and Melnikov.

\begin{theorem}[Downey and Melnikov \cite{DM}]
The lower cone $[{\bf 0, c}]$ is realizable iff ${\bf c}$ is $\textup{low}_2$ or ${\mathbf 0'}$.
\end{theorem}

In the present paper we will consider similar questions for the 
what are arguably the most important \pz classes in terms of applications in 
Reverse Mathematics; the \emph{separating classes}. Recall that $P$ is a 
separating class if there exist c.e.\ disjoint sets $A$, $B$ such that 
$P=S(A,B)=\{Z\mid Z\supseteq A\land Z\cap B=\emptyset\}$.
For example, if $A$ represents the sentences provable in PA and $B$ the sentences refutable,
then $S(A,B)$ represents the complete extensions of PA.
Also, as is well known, if we desire to show that a theorem of second order 
arithmetic is as strong as WKL$_0$, 
then it suffices to code in separating classes,
in spite of the fact that Weak K\"onig's Lemma says that 
infinite binary trees have paths.
That is, we don't have to look at all infinite binary trees, only separating 
classes. For example, it is known that the theorem ``Every countable 
commutative ring with identity has a prime ideal'' is equivalent to 
WKL$_0$, and this is proven by Friedman, Simpson and Smith \cite{FSS1}
by coding a given separating class as the prime ideal structure of a 
commutative ring with identity, but their earlier claim (\cite{FSS})
that, given a \pz class $P$ there is a (computable) commutative ring 
with identity whose prime ideals are in 1-1 correspondence with the members of the class remains open.
We refer the reader to e.g. Downey and Hirschfeldt \cite{DH} or Simpson 
\cite{S}. 
The questions of what $\WW[S(A,B)]$
can be for various $A$, $B$ was asked in both \cite{CDN} and \cite{DM}.

\subsection{Our Results}  
Whilst an exact characterization is elusive, we do prove 
some results to show how different the situation is for separating classes.

We do know that strange ``separation sprectra''  can occur. 
For example, Jockusch and Soare \cite{JS}
showed that there two pairs of  pairwise disjoint c.e.\ sets $A_1,B_1, A_2,B_2$ 
such that every (not necessarily c.e.) separator 
of $A_1$ and $B_1$ is Turing incomparible with every separator of
$A_2$ and $B_2$.
This result was extended by Downey, Jockusch and Stob \cite{DJS}
who showed that four  such sets can be 
below a c.e.\ degree ${\bf a}$ iff ${\bf a}$ is array noncomputable.
Moreover, Jockusch and Soare \cite{JSjsl} extended their earlier 
result to construct 
the sets so that every separator of $A_1,B_1$ forms a minimal 
pair with each separator of $A_2,B_2$. This result was shown to 
be realizable below each \emph{promptly} array noncomputable
c.e.\ degree by Downey and Greenberg \cite{DGbook}.

The basic result used in the arguments in~ \cite{CDN} is that 
any c.e.\ singleton can be realized as a spectrum. 
Then, the authors use some computability theory 
approximation arguments for $\Sigma_3^0$-representable sets to 
get one direction of the characterization.

We do know  of two singletons which can be realized. One is the trivial one
${\mathbf 0}$ if we had a computable $A$ and considered $S(A,\overline{A})$,
but in the nontrivial situation where $\N-(A\sqcup B)$ is infinite,
the only  singleton we know of 
is the PA-one, namely~${\mathbf 0'}$. 
The following question is open:

\begin{question} Is any other singleton ${\mathbf a}$ 
realizable as $\WW(S(A,B))$?
\end{question}

We conjecture that the answer is ``no''.

We can show that the answer is no in the case that $A\equiv_{wtt} B$,
where $\le_{wtt}$ denotes
weak truth table reducibility.
In fact, in \S \ref{comp-2}, we prove 
the following ``upward closure'' result.

\begin{theorem}
Suppose $A \equiv_{wtt} B$ are c.e.\ sets such that:
\begin{itemize}
\item $A \cap B = \emptyset$; and
\item $|\w \setminus (A \cup B)| = \infty$.
\end{itemize}
Then for every $C \ge_T A$, there is a separator of $A$ and $B$ of the same Turing degree as $C$.
\end{theorem}

On the other hand, we really do need ''$A \equiv_{wtt} B$''
in the hypothesis, to force upward closure, as $A\le_{wtt} B$ is not enough, as we see in the next result.

\begin{theorem}
There are c.e.\ sets $A \ge_{wtt} B$ such that:
\begin{itemize}
\item $A \cap B = \emptyset$;
\item $|\omega \setminus (A \cup B)| = \infty$; and
\item No separator of $A$ and $B$ computes $\emptyset'$.
\end{itemize}
\end{theorem}

One of the most natural ways to possibly get a singleton would be to have 
$A\equiv_T B$ and such that 
$S(A,B)$ was  highly constrained in that members of c.e.\ degree $X$ would 
be ``close'' to $A$ and $B$. One place where this idea was used is where 
``c.e.\ degree'' was replaced by ``c.e.'' in, for example, Downey 
\cite{Dowthesis}. There $(A,B)$ is called a \emph{maximal pair} 
if whethever $X$ is a c.e.\ set separating $A$ and $B$ then 
either $X-A$ or $X-B$ is finite (also see Muchnik \cite{Muc}). Any simple c.e. set can be split into a 
maximal pair. They are quite useful in reverse mathematics as, for example,
in \cite{DLM}.
Thus it would be very nice if there was a a stronger version 
of maximal pair with c.e.\ \emph{set} replaced by 
set of c.e.\ \emph{degree}.
\begin{definition}
Two c.e.\ sets $A$ and $B$ form a {\em super-maximal pair} if the following hold:
\begin{itemize}
\item $A \cap B = \emptyset$;
\item $|\w - (A \cup B)| = \infty$; and
\item If $X$ is of c.e.\ degree with $A \subseteq X$ and $B \subseteq \overline{X}$, then $X =^* A$ or $\overline{X} =^* B$.
\end{itemize}
\end{definition}

Alas, no such pairs exist.

\begin{theorem}
Super maximal pairs do not exist.
\end{theorem}

We believe that this result is of independent interest aside from 
our interest in degrees of members of separating classes. The 
proof is surprisingly difficult and requires \emph{three} levels of 
nonuniformity, in the same way that the Lachlan Nondiamond Theorem 
(Lachlan \cite{La66})
needs one level of nonuniformity.
The only other example where exactly 3 levels are needed occurs in
an unpublished manuscript of Slaman where he shows that 
there is a c.e.\ degree ${\mathbf a}\ne {\mathbf 0}$ 
which is not the top of a diamond lattice in the Turing degrees.
Thus this proof is of some technical interest.

Giving up on \emph{one} degree, we ask whether \emph{two}
degrees are possible. This time the answer is yes, provided that one 
is ${\mathbf 0'}$.
One easy way to see this is to take $A$ a complete c.e. set with 
$\overline{A}$ introreducible. Then  $S(A,\emptyset)$ has spectrum
${\bf 0',0}$. The next result shows that ${\bf 0}$ can be replaced by any c.e.\
degree.

\begin{theorem}
For every c.e.\ degree $\mathbf{c}$, the separating spectrum $\{\mathbf{c}, \mathbf{0}'\}$ is possible.
\end{theorem}

We remark that we are unaware of any other definite spectrum which can be realized,
even for the two degree case.

\section{Wtt-results}\label{comp-2}

In this section we prove the comparibility results about weak truth table 
reducibility for $A$ and $B$.
The first shows that we can have no upward closure whilst having comparibility.

\begin{theorem}
There are c.e.\ sets $A \ge_{wtt} B$ such that:
\begin{itemize}
\item $A \cap B = \emptyset$;
\item $|\omega \setminus (A \cup B)| = \infty$; and
\item No separator of $A$ and $B$ computes $\emptyset'$.
\end{itemize}
\end{theorem}

\begin{proof}
We construct such sets.  To achieve $A \ge_{wtt} B$, we promise to never enumerate an element into $B$ unless we simultaneously enumerate a smaller element into $A$.

We build an auxiliary c.e.\ set $D$ and meet the following requirements:
\begin{itemize}
\item[$N_k$:] $(\exists x > k) [x \not \in (A \cup B)]$.
\item[$R_e$:] For any separator $Z \in S(A,B)$, $\Phi_e^Z \neq 
{D}$.
\end{itemize}
Clearly this will suffice.

\smallskip

{\em Strategy for $N_k$:} Wait for a stage $s > k$.  By construction, $s \not \in (A_s \cup B_s)$.  Restrain $(A\cup B)\uhr{s+1}$.

\smallskip

{\em Strategy for $R_e$:} We fix a restraint $r$ to be the stage at which the strategy was last initialized, such that the strategy will not be permitted to enumerate elements below $r$ into $A \cup B$.  The strategy repeats the following loop:
\begin{enumerate}
\item Claim a large $n$ not yet claimed by any strategy.
\item At stage $s$, search for a $\sigma \in 2^s$ such that:
\begin{itemize}
\item $\Phi_e^{\sigma}\uhr{n+1}[s] = D_s\uhr{n+1}$; and
\item For all $x < |\sigma|$, $x \in A_s \to \sigma(x) = 1$, and $x \in B_s \to \sigma(x) = 0$.
\end{itemize}
\item Having found such a $\sigma$, if there is an $m < |\sigma|$ such that $m \ge r$, $\sigma(m) = 1$ and $m \not \in A_s$, fix the least such.  For all $x \in [m, |\sigma|)$, if $\sigma(x) = 1$, then enumerate $x$ into $A_{s+1}$, and if $\sigma(x) = 0$, then enumerate $x$ into $B_{s+1}$.
\item Regardless of whether a desired $m$ exists, enumerate $n$ into $D$ and return to Step (1).
\end{enumerate}

\smallskip

{\em Construction:}  Arrange the strategies into a priority ordering.  At stage $s$, run the first $s$ strategies, in order of priority.  Whenever a strategy acts, initialize all lower priority strategies.

\smallskip

{\em Verification:} By construction, we never enumerate a number into~$B$ unless we simultaneously enumerate a smaller number into~$A$, and so $B \le_{wtt} A$.  Also, we never enumerate a number into $(A \cup B)[s+1]$ unless that number is smaller than~$s$.

\begin{claim}
Suppose the $R_e$-strategy is only initialized finitely many times.  Then it only enumerates finitely many numbers into~$D$.
\end{claim}

\begin{proof}
Let $r$ be the final restraint imposed on the strategy.  Towards a contradiction, suppose there are infinitely many stages at which the strategy enumerates an element into $D$, and list those which occur after the final time the strategy was initialized as $s_0 < s_1 < \dots$.  Fix $n_i$ the element enumerated at stage $s_i$, $\sigma_i$ the witnessing $\sigma$, and let $m_i$ be the selected $m$, if it exists, and otherwise let $m_i = s_i$.  By well-ordering properties, there are infinitely many $i$ such that for all $j > i$, $m_i \le m_j$.  Fix such an $i$.

Fix a $j > i$.  By construction, we have $n_i < n_{j}$.  So it cannot be that $\sigma_{j}$ extends $\sigma_i$, as that would give $\Phi_e^{\sigma_{j}}(n_i) = \Phi_e^{\sigma_{i}}(n_i) = 0$, but $n_i \in D_{s_i+1} \subseteq D_{s_{j}}$, contrary to our choice of $\sigma_{j}$.  So there must be some $y < |\sigma_i|$ with $\sigma_i(y) \neq \sigma_{j}(y)$.  By construction, $y < m_i$.

If $y \ge r$, we claim that $\sigma_i(y) = 0$.  For if not, then $\sigma_{i+1}(y) = 0$ implies $y \not \in A_{s_{i+1}} \supseteq A_{s_i}$, and so $y$ contradicts our choice of $m_i$.

So if $y \ge r$, $y \not \in A_{s_i}$.  Higher priority strategies will never act after stage $s_i$, and lower priority strategies were initialized and so have restraints greater than $y$, so neither can enumerate $y$ into $A$.  By choice of $i$, $y < m_i \le m_k$ for all $k \in [i, j]$, and so the action of this strategy cannot enumerate $y$ into $A$ before stage $s_j$.  So $y \not \in A_{s_j}$, and so if $y \ge r$, then $y$ is a viable candidate for $m_j$.  This contradicts $m_i < m_j$.  It follows that $y < r$, and so $\sigma_i\uhr{r} \neq \sigma_j\uhr{r}$, for all $j > i$.

But there are only finitely many strings in $2^r$, and so there can be only finitely many $i$ such that for all $j > i$, $m_i \le m_j$, contrary to our earlier observation.  The claim follows.
\end{proof}

It is now a simple induction to show that each strategy is only initialized finitely many times.  Clearly each $N_k$-strategy ensures its requirement.  Also each $R_e$-strategy must eventually wait forever at Step (2), and so there can be no separator computing $D$.  This completes the proof.
\end{proof}

The second result shows that if there is a realizable singleton which is 
not~${\mathbf 0} $ or~${\mathbf 0'}$, then we cannot use 
non-adaptive reductions.

\begin{theorem}
Suppose $A \equiv_{wtt} B$ are c.e.\ sets such that:
\begin{itemize}
\item $A \cap B = \emptyset$; and
\item $|\w \setminus (A \cup B)| = \infty$.
\end{itemize}
Then for every $C \ge_T A$, there is a separator of $A$ and $B$ of the same Turing degree as $C$.
\end{theorem}

\begin{proof}
Fix wtt-operators $\Gamma$ and $\Delta$ with $\Gamma^A = B$ and $\Delta^B = A$, and let $f$ be a computable function bounding the use of both $\Gamma$ and $\Delta$.  We assume that $f$ is monotonic and $f(x) > x$ for all $x$.  We split the argument into two cases, depending on whether or not there are infinitely many $x$ with $(x, f(x)] \subseteq A \cup B$.

\smallskip

{\em Case 1.}  Suppose there are only finitely many $x$ with $(x, f(x)] \subseteq A \cup B$.

Fix $k$ such that there are no such $x \ge k$.  Define the computable sequence: $m_0 = k$; $m_{i+1} = f(m_i)$.  Then we define a separator $Z$ such that if $n \not \in C$, then $Z$ agrees with $A$ on $(m_n, m_{n+1}]$, and if $n \in C$, then $Z$ agrees with $\overline{B}$ on $(m_n, m_{n+1}]$.  As $C$ computes $A$ (and thus $B$), $Z \le_T C$.

To establish $C \le_T Z$, note that $Z$ cannot agree with both $A$ and $\overline{B}$ on any $(m_n, m_{n+1}]$, as $(m_n, m_{n+1}] \not \subseteq A \cup B$.  Thus $n \in C$ iff $Z$ agrees with $\overline{B}$ on $(m_n, m_{n+1}]$ iff $Z$ differs from $A$ on $(m_n, m_{n+1}]$.  So $Z$ can determine whether $n \in C$ by waiting for a stage $s$ such that it agrees with either $A_s$ or $\overline{B}_s$ on $(m_n, m_{n+1}]$.

\smallskip

{\em Case 2.}  Suppose there are infinitely many $x$ with $(x, f(x)] \subseteq A \cup B$.

Define the following sequence:
\begin{itemize}
\item $m_0 = -1$;
\item $m_{n+1}$ is the least $x > m_n$ such that:
\begin{itemize}
\item $(m_n, x] \not \subseteq A \cup B$; and
\item $(x, f(x)] \subseteq A \cup B$.
\end{itemize}
\end{itemize}
By assumption, $m_n$ exists for all $n$.  We define a separator $Z$ such that if $n \not \in C$, then $Z$ agrees with $A$ on $(m_n, m_{n+1}]$, and if $n \in C$, then $Z$ agrees with $\overline{B}$ on $(m_n, m_{n+1}]$.

First observe that since $C$ computes $A$ (and thus $B$), $C$ computes $(m_n)_{n \in \omega}$.  Thus $C$ computes $Z$.

Next, we argue that $(m_n)_{n \in \omega}$ is computable from $Z$.  Suppose we have determined $m_i$ for $i \le n$.  Then let $s$ be a stage such that the following hold:
\begin{itemize}
\item For every $i < n$, $Z$ on $(m_i, m_{i+1}]$ agrees with either $A_s$ or $\overline{B}_s$; 
\item For every $i < n$, $(m_i, f(m_i)] \subseteq A_s \cup B_s$; and
\item There is some $x > m_n$ such that:
\begin{itemize}
\item $\Gamma^{A_s}\uhr{x+1} = B\uhr{x+1}$ and $\Delta^{B_s}\uhr{x+1} = A\uhr{x+1}$; 
\item $Z$ on $(m_n, x]$ agrees with either $A_s$ or $\overline{B}_s$;
\item $(m_n, x] \not \subseteq A_s \cup B_s$; and
\item $(x, f(x)] \subseteq A_s \cup B_s$.
\end{itemize}
\end{itemize}
Fix $y$ the least such $x$. 

\begin{claim}
$y = m_{n+1}$.
\end{claim}

\begin{proof}
Suppose not.  If $m_{n+1} > y$, then by minimality of $m_{n+1}$, it must be that $(m_n, y] \subseteq A \cup B$.  By choice of $y$, this means that some element $z \le y$ is enumerated into $A \cup B$ after stage $s$.

If $m_{n+1} < y$, then by minimality of $y$, it must be that $(m_{n+1}, f(m_{n+1})] \not \subseteq A_s \cup B_s$.  By monotonicity, $f(m_{n+1}) \le f(y)$, so there must be some element $z \le f(y)$ enumerated into $A \cup B$ after stage $s$.  By choice of $y$, $z \le y$.

In either case, we see that there is an element $z \le y$ enumerated into $A \cup B$ after stage $s$.  By choice of $s$ and correctness of $\Delta$, if $z \in A$, then there must be a $w \le f(z)$ such that $w$ is enumerated into $B$ after stage $s$, and by monotonicity and choice of $s$ and $y$, $w \le y$.  If $z \in B$, then symmetric reasoning shows there is a $w \le y$ enumerated into $A$ after stage $s$.  Without loss of generality, assume that $z$ is the least element enumerated into $A \cup B$ after stage $s$

If there is an $i < n$ such that $z \in (m_i, m_{i+1}]$, then by choice of $s$ and correctness of $\Gamma$ and $\Delta$, $w \in (m_i, m_{i+1}]$.  But $Z$ agrees with either $A_s$ or $\overline{B}_s$ on $(m_i, m_{i+1}]$ by choice of $s$, so there cannot be such $z$ and $w$: e.g.\ if $Z$ agrees with $A_s$, then $z \in A \setminus A_s$ contradicts $Z$ being a separator.  So it must be that $z,w \in (m_n, x]$.  But again, $Z$ agrees with either $A_s$ or $\overline{B}_s$ on $(m_n, x]$, so this is a contradiction.
\end{proof}

Having computed $(m_n)_{n\in \omega}$, $Z$ can determine whether $n \in C$ by waiting for a stage $s$ such that $Z$ on $(m_n, m_{n+1}]$ agrees with either $A_s$ or $\overline{B}_s$ (in fact, the stage $s$ used to determine $m_{n+1}$ suffices).
\end{proof}

\section{No Super-Maximal Pairs}
Recall that super-maximnal pairs are c.e.\ sets $A$, $B$ where 
all $X$ separating of c.e.\ degree must 
be finite variants of 
either $A$ or $B$.

\begin{theorem}
Super maximal pairs do not exist.
\end{theorem}

\begin{proof}
Let $A$ and $B$ be two c.e.\ sets satisfying $A \cap B = \emptyset$ and $|\w - (A\cup B)| = \infty$.  We will build a separating set $X$ of c.e.\ degree such that $X \neq^* A$ and $\overline{X} \neq^* B$.  This construction is necessarily non-uniform, however, so we may make up to three attempts.  If the first attempt fails, it will be because our set has either $X =^* A$ or $\overline{X} =^* B$.  If the second attempt fails, it will fail in the opposite manner.  Then the third attempt will succeed.

\medskip

{\em The First Attempt.} We build a computable sequence $(X_{1, s})_{s \in \omega}$ approximating $X_1$.  We begin with $X_{1, 0} = \emptyset$.  We also define the auxiliary sets $Y^n_{1, 0} = \emptyset$ for all $n$.

At stage $s+1$, we first define a sequence $x^1_{-1,s+1} < x^1_{0, s+1} < \dots < x^1_{k, s+1} = s$:
\begin{itemize}
\item $x^1_{-1,s+1} = -1$.
\item Given $x^1_{n, s+1} < s$, if $x^1_{n+1,s}$ is undefined or $x^1_{n+1,s} \le x^1_{n,s+1}$, we let $x^1_{n,s+1} = s$.
\item Given $x^1_{2i,s+1} < x^1_{2i+1,s} < s$, if there is a $y \in (x^1_{2i,s+1},x^1_{2i+1,s}]$ satisfying:
\begin{itemize}
\item $y \in X_{1,s} - (A_{s+1} \cup B_{s+1})$; or
\item $y \not \in A_{s+1} \cup B_{s+1}$, and there is some $z < y$ with $z \in X_{1,s} \cap B_{s+1}$ or with $z \in \overline{X}_{1,s} \cap A_{s+1}$; or
\item $y \not \in A_{s+1} \cup B_{s+1}$ and $y = s$,
\end{itemize}
then we let $x^1_{2i+1,s+1} = x^1_{2i+1,s}$.  Otherwise, we let $x^1_{2i+1,s+1} = s$.
\item Given $x^1_{2i+1,s+1} < x^1_{2i+2,s} < s$, if there is a $y \in (x^1_{2i+1,s+1},x^1_{2i+2,s}]$ satisfying:
\begin{itemize}
\item $y \in \overline{X}_{1,s} - (A_{s+1} \cup B_{s+1})$; or
\item $y \not \in A_{s+1} \cup B_{s+1}$, and there is some $z < y$ with $z \in X_{1,s} \cap B_{s+1}$ or with $z \in \overline{X}_{1,s} \cap A_{s+1}$; or
\item $y \not \in A_{s+1} \cup B_{s+1}$ and $y = s$,
\end{itemize}
then we let $x^1_{2i+2,s+1} = x^1_{2i+2,s}$.  Otherwise, we let $x^1_{2i+2,s+1} = s$.
\end{itemize}
Note that the sequence we define is strictly increasing.  It terminates when it achieves $s$.

On each interval $(x^1_{2i,s+1}, x^1_{2i+1,s+1}]$, we will attempt to ensure that $X \neq^* A$ by putting whatever elements we can into $X_1$.  On each interval $(x^1_{2i+1,s+1}, x^1_{2i+2,s+2}]$, we will do the opposite.  Of course, this is restricted by our need to make $X_1$ a separator of c.e.\ degree.  Accordingly, we make the following definition
\begin{definition}
We say that $y$ is {\em permitted at stage $s+1$} (for $X_1$) if $y \not \in A_{s+1} \cup B_{s+1}$, and $y = s$ or there is some $z < y$ with $z \in X_{1,s} \cap B_{s+1}$ or with $z \in \overline{X}_{1,s} \cap A_{s+1}$.
\end{definition}

We now define $X_{1,s+1}$ as follows:
\begin{itemize}
\item If $y \in A_{s+1}$, then $y \in X_{1,s+1}$.
\item If $y \in B_{s+1}$, then $y \not \in X_{1,s+1}$.
\item If $y \in (x^1_{2i,s+1}, x^1_{2i+1,s+1}]$ and is permitted at stage $s+1$, then $y \in X_{1,s+1}$.
\item If $y \in (x^1_{2i+1,s+1}, x^1_{2i+2,s+1}]$ and is permitted at stage $s+1$, then $y \not \in X_{1,s+1}$.
\item If none of the above apply, then $x \in X_{1,s+1} \iff x \in X_{1,s}$.
\end{itemize}
This completes the first construction.

\begin{claim}\label{claim:first_no_super_max_claim}
$(X_{1, s})_{s \in \omega}$ converges (in a $\Delta^0_2$ fashion) to a set $X_1$, and $X_1$ is a separator of $A$ and $B$ of c.e.\ degree.
\end{claim}

\begin{proof}
Observe first that if $x \in X_{1,s} \triangle X_{1,s+1}$, then one of the following must hold:
\begin{enumerate}
\item\label{no_super_max:first_permission} $x \in X_{1,s} \cap B_{s+1}$;
\item\label{no_super_max:second_permission} $x \in \overline{X}_{1, s} \cap A_{s+1}$;
\item For some $z < x$, one of (\ref{no_super_max:first_permission}) or (\ref{no_super_max:second_permission}) holds; or
\item $x = s$.
\end{enumerate}
By induction on $x$, each of these can occur only finitely many times for each $x$, and so the limit $X_1$ exists.

Now let $W = \{ x : \exists s\, x \in X_{1,s} \cap B_{s+1} \vee x \in \overline{X}_{1, s} \cap A_{s+1}\}$ with the obvious c.e.\ approximation $(W_s)_{s \in \omega}$.  If $x \in W_{s+1} - W_s$, then $x \in X_{1, s}\triangle X_{1, s+1}$.  Conversely, if $x \in X_{1, s}\triangle X_{1, s+1}$ for some $s > x$, then there is a $z \le x$ with $z \in W_{s+1} - W_s$.  Thus $X_1 \equiv_T W$, and so $X_1$ is of c.e.\ degree.

That $A \subseteq X$ and $X \cap B = \emptyset$ is immediate from our definition of $X_{1, s+1}$.
\end{proof}

Now, for each $n$, we consider the sequence $(x^1_{n, s})_{s \in \omega}$.  Note that not every $x^1_{n, s}$ will be defined.  However, if $x^1_{n, s}$ is defined for almost every $s$, we can consider whether the sequence has a limit.

\begin{claim}\label{claim:second_no_super_max_claim}
Suppose $x^1_n = \lim_s x^1_{n, s}$ exists with $x^1_n < \infty$ (and, implicitly, $x^1_{n, s}$ is defined for almost every $s$).  Then for every $j < n$, $x^1_j = \lim_s x^1_{j,s}$ exists, and $x^1_j < x^1_n$.  Further, if $n > -1$, then there is a $y \in (x^1_{n-1}, x^1_n]$ with $y \not \in A \cup B$, and $y \in X_1$ if and only if $n$ is odd.
\end{claim}

\begin{proof}
By construction, for $j < n$ and every $s$ at which $x^1_{n,s}$ is defined, $x^1_{j, s}$ is also defined.  So $x^1_{j, s}$ is defined for almost every $s$.  By construction, $x^1_{j,s}$ is nondecreasing in $s$ and bounded by $x^1_n$, so $x^1_j < x^1_n$ exists.

For $s$ with $x^1_{n,s+1} = x^1_n < s$, there is some $y$ witnessing that $x^1_{n,s+1} \neq s$.  This $y$ is in $X_{1,s+1}$ if and only if $n$ is odd by construction.  So by pigeon hole, there is some $y$ in $(x^1_{n-1}, x^1_n]$ with this property for infinitely many $s$, and thus for almost every $s$ (as the approximation to $X_1$ converges).  Thus $y \in X_1$ if and only if $n$ is odd.
\end{proof}

So if $x^1_n < \infty$ exists for every $n$, then $X_1$ is as desired.  So instead assume $\ell_1$ is the greatest $n$ such that $x^1_n < \infty$ exists.  Let $k_1 = \ell_1+1$, and WLOG assume $k_1$ is odd.  Observe that $x^1_{k_1, s+1}$ is defined for every $s > x^1_{\ell_1}$.

\begin{claim}\label{claim:third_no_super_max_claim}
For every $s$ with $x^1_{\ell_1,s} = x^1_{\ell_1}$ and $x^1_{k_1, s}$ defined, and for every $y \in (x^1_{\ell_1}, x^1_{k_1, s}] \cap X_{1, s}$, $y \in A \cup B$.
\end{claim}

\begin{proof}
Fix $t \ge s$ such that $x^1_{k_1, t+1} \neq x^1_{k_1, t} = x^1_{k_1, s}$.  Then, by definition of $x^1_{k_1, t+1}$, every such $y$ must be in $A_{t+1} \cup B_{t+1} \cup \overline{X}_{1, t}$.  But, by induction on $s' \in (s, t]$, if $y \not \in B$, then $y \in X_{1, s'}$.
\end{proof}

It follows that for any $y > x^1_{\ell_1}$, $y \in X_1 \iff y \in A$.

\smallskip

{\em The Second Attempt.}  We make a second attempt based on the knowledge of how the first attempt failed.  First, we perform a speedup of the enumerations of $A$ and $B$ and of the previous construction such that the following all hold:
\begin{itemize}
\item For each $n < k_1$ and every $s$, $x^1_{n, s} = x^1_n$.
\item For every $s$, $x^1_{k_1, s} > s$ is defined.
\item For every $y \in (x^1_{\ell_1}, s]$, $y \in A_s \cup \overline{X}_{1,s}$.
\end{itemize}

We now build $X_2$ in this new timeline.  We simply repeat the construction of $X_1$, except that we refer to the sequence we build at each stage as $(x^2_{n,s+1})$, to avoid confusion, and we always begin with $x^2_{-1, s+1} = x^1_{\ell_1}$.  Analogues of Claims \ref{claim:first_no_super_max_claim} and \ref{claim:second_no_super_max_claim} for $X_2$ proceed as the originals.  Note that for every $n$ and $s$ with $x^2_{n,s}$ defined, $x^2_{n,s} < x^1_{k_1,s}$.

Again, if $x^2_n < \infty$ exists for every $n$, then $X_2$ is our desired separator.  So instead assume $\ell_2$ is the greatest $n$ such that $x^2_n < \infty$ exists, and let $k_2 = \ell_2 + 1$.  Again, $x^2_{k_2,s+1}$ is defined for every $s > x^2_{\ell_2}$.

\begin{claim}\label{claim:no_super_max_even}
$k_2$ is even.
\end{claim}

\begin{proof}
Suppose not.  Fix an $s > x^2_{\ell_2} > x^1_{\ell_1}$ with $s \not \in A \cup B$.

Fix $t$ least with $x^2_{k_2, t+1} \ge s$.  We claim that $s$ is permitted for $X_2$ at stage $t+1$.  This is immediate if $s = t$, so suppose not.  Then $x^2_{k_2, t} < s < t$, so there is some $y \in (x^2_{\ell_2}, x^2_{k_2,t}] \cap X_{2,t}$ witnessing that $x^2_{k_2, t} < t-1$.  As this $y$ does not suffice for $t+1$, it must be that $y \in A_{t+1} \cup B_{t+1}$.  In second case, $y$ witnesses that $s$ is permitted for $X_2$ at stage $t+1$.  In the first case, since $y \not \in A_t$, $y \in \overline{X}_{1, t}$.  So $y$ witnesses that $s$ is permitted for $X_1$ at some point in the time period between stages $t$ and $t+1$ of our speedup.  But as $x^1_{\ell_1} < s < x^1_{k_1, s} \le x^1_{k_1, t}$, a definition $s \in X_1$ would be made.  This contradicts Claim~\ref{claim:third_no_super_max_claim}.

So $s$ is permitted for $X_2$ at stage $t+1$, and so $s \in X_{2, t+1}$.  But then $s$ will forever witness that $x^2_{k_2, s'}$ does not need to change, contrary to choose of $k_2$.
\end{proof}

Analogously to Claim~\ref{claim:third_no_super_max_claim}, we obtain the following.

\begin{claim}\label{claim:fourth_no_super_max_claim}
For every $s$ with $x^2_{\ell_2, s} = x^2_{\ell_2}$ and $x^2_{k_2,s}$ defined, and for every $y \in (x^2_{\ell_2}, x^2_{k_2, s}] \cap \overline{X}_{2, s}$, $y \in A \cup B$.
\end{claim}

It follows that for any $y > x^2_{\ell_2}$, $y \in X_2 \iff y \not \in B$.

\smallskip

{\em The Final Attempt.}  We make our final attempt based on the knowledge of how the first two attempts failed.  Again, we begin with a speedup of our previous speedup such that the following all hold:
\begin{itemize}
\item For each $n < k_2$ and every $s$, $x^2_{n, s} = x^2_n$.
\item For every $s$, $x^2_{k_2, s} > s$ is defined.
\item For every $y \in (x^2_{\ell_2}, s]$, $y \in X_{2, s} \cup B_s$.
\end{itemize}

Again, we always begin with $x^3_{-1, s+1} = x^2_{\ell_2}$.  We again prove analogues of Claims \ref{claim:first_no_super_max_claim} and \ref{claim:second_no_super_max_claim}.  Define $k_3$ to be the least such that $x^3_{k_3} < \infty$ does not exist.  We show two versions of Claim~\ref{claim:no_super_max_even}, one showing that $k_3$ cannot be odd by arguing to a contradiction with Claim~\ref{claim:third_no_super_max_claim}, and another showing that $k_3$ cannot be even by arguing to a contradiction with Claim~\ref{claim:fourth_no_super_max_claim}.  It follows that there is no such $k_3$, and so $X_3$ is as desired.

\end{proof}

\section{Two degrees}

In this section we examine degree spectra containing two degrees.

\begin{theorem}
For every c.e.\ degree $\mathbf{c}$, the separating spectrum $\{\mathbf{c}, \mathbf{0}'\}$ is possible.
\end{theorem}

\begin{proof}
Fix $C \in \mathbf{c}$.  We build disjoint c.e.\ sets $A$ and a $B$ with $|\omega \setminus (A \cup B)| = \infty$ and meeting the following requirements, for all $e$ and $n$:
\begin{itemize}
\item[$R_e$:] If $\Phi_e^{W_e} = Z$, $A \subseteq Z$, 
and $|\overline{Z} \setminus B| = \infty$, then $W_e$ computes $K$.
\item[$P_n$:] There is at most one $i$ with $\seq{n, i} \in B$, and such $i$, if it exists, is bounded by $n^2+1$.  Further, $n \in C \iff \exists i\, [\seq{n, i} \in B]$.
\end{itemize}

First we argue that meeting these requirements suffices.

By the $P_n$, $B \equiv_T C$.  For one direction, $n \in C \iff (\exists i \le n^2+1) [\seq{n, i} \in B]$, which is bounded quantification.  For the other, suppose we wish to determine whether $\seq{n, j}$ is in $B$.  We first check whether $n \in C$; if not, $\seq{n, j} \not \in B$.  If so, we enumerate $B$ until we see some $\seq{n, i} \in B$.  Then $\seq{n, j} \in B \iff i = j$.

Now, suppose $Z$ is a separator of c.e.\ degree.  If $|\bar{Z} \setminus B| = \infty$, then by the appropriate $R_e$, $Z$ is of degree $\mathbf{0}'$.  If instead $\bar{Z} =^* B$, then $Z$ is of degree $\mathbf{c}$.  In particular, $A$ itself is a separator of degree $\mathbf{0}'$, and $\overline{B}$ is a separator of degree $\mathbf{c}$.

We assume that for all $i$, $\seq{n, i} \ge n^3$.

\smallskip

{\em Strategy for $R_e$:}  We construct a c.e. operator $V_e$, with the intention that $V_e^{W_e} = \overline{K}$ if $\Phi_e^{W_e}$ is as described.  For each $m$, if $m \not \in [K \cup V_{e}^{W_{e}}][s]$, we search for an $x$ and a $\gamma$ satisfying the following:
\begin{itemize}
\item For all $n \le \max\{e, m\}$ and $i < n^2+1$, $x > \seq{n, i}$.
\item $x \not \in B_s$;
\item For all $y \le x$, $\Phi_e^{W_e\uhr{\gamma}}(y)[s]\converge$; and
\item $\Phi_e^{W_e\uhr{\gamma}}(x)[s] = 0$.
\end{itemize}
If these exist, we fix the least such $x$ and $\gamma$ (by standard use assumptions, we can minimize these both simultaneously), and we define $m \in V_e^{W_e\uhr{\gamma}}[s]$.  At every subsequent stage $t > s$, if $W_e\uhr{\gamma}[t] = W_e\uhr{\gamma}[s]$, then $x$ is {\em blocked from $B$} at stage $t$.

If instead $m \in [K \cap V_e^{W_e}][s]$, we fix the $x$ used in the definition of $m \in V_e^{W_e}[s]$.  If $x \not \in B_s$ (as we will later argue must be the case), we enumerate $x$ into $A$.

Otherwise, do nothing for $m$.

\smallskip

{\em Strategy for $P_n$:}  When $n$ is enumerated into $C$, fix the least $i$ such that $\seq{n, i} \not \in A_s$ and $\seq{n,i}$ is not blocked from $B$ at stage $s$ by any of the $R_e$.  We will later argue that $i < n^2+1$.  Enumerate $\seq{n, i}$ into $B$.

\smallskip

{\em Construction:}  Simply run all of the above strategies simultaneously.  We take no care to ensure that different $R$-strategies are choosing different witnesses $x$.

\smallskip

{\em Verification:}  By construction, we never enumerate an element into $A \cup B$, and so $A$ and $B$ are disjoint.  Next, we show that our claim in the $P_n$-strategy holds.

\begin{claim}
At any stage $s$, for every $n$,
\[
\# \{ i < n^2+1: \seq{n, i} \in A_s \vee \text{$\seq{n,i}$ is blocked from $B$ at stage $s$}\} \le n^2.
\]
\end{claim}

\begin{proof}
The only way for $x = \seq{n, i}$ to be blocked or enumerated into $A$ is for it to be selected by some $R_e$-strategy on behalf of some $m$.  By our choice of such $x$ in the $R_e$-strategy, it must be that $e, m < n$.  Further, no pair $(e, m)$ can contribute more than one $x$ in this fashion: the $R_e$-strategy blocks at most one $x$ at a time on behalf of each $m$, and if the strategy has enumerated an $x$ on behalf of $m$, then $m \in C_s$ and so the strategy will not block or enumerate another element on behalf of $m$.

The claim follows.
\end{proof}

Thus our $P_n$-strategy meets its requirement.

\begin{claim}
For all $k$,
\[
|(A \cup B) \cap k^3| \le O(k^2).
\]
It follows that $|\omega \setminus (A \cup B)|$ is infinite.
\end{claim}

\begin{proof}
By construction, since $\seq{k,0} \ge k^3$, the only strategies which can enumerate elements into $A \cap k^3$ are $R_e$-strategies with $e < k$, and then only on behalf of some $m < k$.  As in the previous claim, each pair $(e, m)$ can contribute at most one such enumeration, and so $|A \cap k^3| \le k^2$.

Also, the only strategies which can enumerate elements into $B \cap k^3$ are $P_n$-strategies with $n < k$, and each necessarily enumerates at most one element.  So $|B \cap k^3| \le k$.

The claim follows.
\end{proof}

\begin{claim}
Each $R_e$-strategy meets its requirement.
\end{claim}

\begin{proof}
Fix $m$.  Suppose first that $m \not \in K$.  Fix the least $x \not \in (Z \cup B)$ such that $x > \seq{n, i}$, for all $n \le \max\{e, m\}$ and $i < n^2+1$.  Fix $\gamma$ least such that $\Phi_e^{W_e\uhr{\gamma}} \supseteq Z\uhr{x+1}$, and fix $s_0$ such that $W_e\uhr{\gamma} = W_{e, s_0}\uhr{\gamma}$.  Then at any stage $s > s_0$ with $m \not \in V_e^{W_e}[s]$, $x$ and $\gamma$ will be chosen for the new definition of $m \in V_e^{W_e\uhr{\gamma}}[s]$.  By choice of $s_0$, this ensures $m \in V_e^{W_e}$.

If instead $m \in K$, then fix $s_0$ least with $m \in K_s$.  By construction, we will never define $m \in V_e^{W_e}[s]$ for any $s \ge s_0$.  So suppose $m \in V_e^{W_e}[s_0]$ because of our action at some stage $t < s_0$.  Fix the witnessing $x$ and $\gamma$.  Then necessarily, $W_{e,r}\uhr{\gamma} = W_{e, s_0}\uhr{\gamma}$ for every $r \in [t, s_0]$, and so $x$ was blocked from $B$ at every such stage.  Also, $x \not \in B_t$.  By construction, $x \not \in B_{s_0}$.  So $x$ will be enumerated into $A$ at stage $s_0$.

By assumption, $\Phi_e^{W_e\uhr{\gamma}}(x)[t] = \Phi_e^{W_e\uhr{\gamma}}(x)[s_0] = 0$.  So if $\Phi_e^{W_e} = Z$ contains $A$, and in particular contains $x$, it must be that $W_{e,s_0}\uhr{\gamma} \neq W_e\uhr{\gamma}$.  Since no future stage $s$ will define $m \in V_e^{W_e}[s]$, it follows that $m \not \in V_e^{W_e}$.
\end{proof}

This completes the proof.
\end{proof}


\end{document}